\documentclass{article}

\usepackage{amssymb,amsmath,color}
\usepackage{amsthm}
\usepackage{url}
\usepackage{tikz}

\definecolor{red}{rgb}{1,0,0}
\definecolor{blue}{rgb}{.2,.2,.8}

\def\l{\lambda}
\def\m{\mu}

\def\A{\mathcal A}

\newtheorem{theorem}{Theorem}[section]
\newtheorem{corollary}[theorem]{Corollary}

\theoremstyle{definition}

\newcommand*\samethanks[1][\value{footnote}]{\footnotemark[#1]}

\begin{document}

\title{Bisected theta series, least $r$-gaps in partitions,  and polygonal numbers}
\author{Cristina Ballantine\thanks{This work was partially supported by a grant from the Simons Foundation (\#245997 to Cristina Ballantine).}\\
	\footnotesize Department of Mathematics and Computer Science\\
	\footnotesize College of The Holy Cross\\
	\footnotesize Worcester, MA 01610, USA \\
	\footnotesize cballant@holycross.edu
	\and Mircea Merca\samethanks
	\\ 
	\footnotesize Academy of Romanian Scientists\\
	\footnotesize Splaiul Independentei 54\\
	\footnotesize Bucharest, 050094 Romania\\
	\footnotesize mircea.merca@profinfo.edu.ro
}
\date{}
\maketitle

\begin{abstract} The least $r$-gap, $g_r(\lambda)$, of a partition $\lambda$ is the smallest part of $\lambda$ appearing less than $r$ times. In this article we introduce two new partition functions involving  least $r$-gaps. 
We consider a bisection of a classical theta identity  and prove new identities relating Euler's partition function $p(n)$, polygonal numbers, and the new partition functions. To prove the results we use an interplay of combinatorial and $q$-series methods. 

 We also give a combinatorial interpretation for
$$\sum_{n=0}^\infty (\pm 1)^{k(k+1)/2} p(n-r\cdot k(k+1)/2).$$
\end{abstract}

\section{Introduction}

In \cite{Merca15}, the second author considered a bisection of Euler's pentagonal number theorem
\begin{equation*}
(q;q)_\infty = \sum_{k=0}^\infty (-1)^{\lceil k/2 \rceil} q^{G_k}
\end{equation*}
based on the parity of the $k$-th generalized pentagonal number
$$G_{k} = \frac{1}{2} \left\lceil \frac{k}{2} \right\rceil \left\lceil \frac{3k+1}{2} \right\rceil,$$
and obtained the following result:
\begin{equation}\label{bisection1}
\sum_{k=0}^{\infty}\frac{1+(-1)^{G_k}}{2}(-1)^{\lceil k/2 \rceil}q^{G_k}=(q^2,q^{12},q^{14},q^{16},q^{18},q^{20},q^{30},q^{32};q^{32})_{\infty},
\end{equation}
where
	$$(a_1,a_2\ldots,a_n;q)_{\infty}=(a_1;q)_{\infty}(a_2;q)_{\infty}\cdots(a_n;q)_{\infty}.$$
Because the infinite product 
$$(a;q)_\infty = \prod_{k=0}^{\infty} (1-aq^k)$$
diverges when $a\neq 0$ 
and $|q|\geqslant 1$, 
whenever $(a;q)_\infty$ appears in a formula, we shall assume that $|q| <1$.
	
The following identity for Euler's partition function $p(n)$ was obtained in \cite{Merca15} as a combinatorial interpretation of \eqref{bisection1}:
\begin{equation}
\sum_{k=0}^\infty \frac{1+(-1)^{G_k}}{2}(-1)^{\lceil k/2 \rceil} p(n-G_k) = L(n),
\end{equation}
where $L(n)$ is the number of partitions of $n$ into parts not congruent to 
$0$, $2$, $12$, $14$, $16$, $18$, $20$ or $30 \mod{32}$. This identity is a bisection of Euler's  well-known recurrence relation for the partition function $p(n)$:
\begin{equation}\label{ER}\sum_{k=0}^\infty (-1)^{\lceil k/2 \rceil} p(n-G_k) = \delta_{0,n},\end{equation}
where $\delta_{i,j}$ is the Kronecker delta function. For  details on \eqref{ER} see Andrews's book \cite{Andrews76}.

In this paper, motivated by these results, we  consider a bisection of another classical theta identity \cite[eq. 2.2.13]{Andrews76}
\begin{equation}\label{eq1}
\frac{(q^2;q^2)_\infty}{(-q;q^2)_\infty} = \sum_{k=0}^{\infty} (-q)^{k(k+1)/2}
\end{equation}
in order to derive new identities for Euler's partition function. These identities involve new partition functions which we define below.

For what follows, we denote by $g_r(\l)$  the smallest part of the partition $\l$ appearing less than $r$ times. 
	\begin{table}[t]
		\centering
		\begin{tabular}{c c c c c c c c} 
			\hline
			$\lambda$ &  5 & 4+1 & 3+2 & 3+1+1 & 2+2+1 & 2+1+1+1 & 1+1+1+1+1 \\ [0.5ex] 
			$g_1(\lambda)$  & 1 & 2 & 1 & 2 & 3 & 3 & 2 \\ [0.5ex] 
			$g_2(\lambda)$  & 1 & 1 & 1 & 2 & 1 & 2 & 2 \\ [0.5ex] 
			$g_3(\lambda)$  & 1 & 1 & 1 & 1 & 1 & 2 & 2 \\ [0.5ex] 
			$g_4(\lambda)$  & 1 & 1 & 1 & 1 & 1 & 1 & 2 \\ [0.5ex] 		
			$g_5(\lambda)$  & 1 & 1 & 1 & 1 & 1 & 1 & 2 \\ [0.5ex] 		
			$g_6(\lambda)$  & 1 & 1 & 1 & 1 & 1 & 1 & 1 \\ [0.5ex] 				
			\hline
		\end{tabular}
		\label{table:1}
		\caption{The partition functions $g_r$ for $\lambda\vdash 5$}
	\end{table}		
The limit distribution  of $g_r(\l)$ has been studied in \cite{Wagner}. In the literature, $g_1(\l)$ is referred to as the least gap of $\l$. By analogy, we refer to $g_r(\lambda)$ as the \textit{least $r$-gap} of $\lambda$. To make formulas more concise, we set $g_0(n)=\infty$. 
We denote by $S_r(n)$ the sum of the least $r$-gaps in all partitions of $\lambda$, i.e., $$S_r(n)=\sum_{\l \vdash n}g_r(\l).$$ Thus, $S_1(n)$ is the sum of the least gaps in all partitions of $n$. By Table \ref{table:1}, we see, for example, that
$$S_1(5) = 1+2+1+2+3+3+2 = 14$$ and $$S_4(n)=1+1+1+1+1+1+2=8.$$

When $r \geqslant 2$, for each partition $\l$ we have $g_r(\l)\leqslant g_{r-1}(\l)$. Let $G_r(n)$ be the number of partitions $\l$ of $n$ satisfying $g_r(\l)<g_{r-1}(\l)$. It is clear that $G_1(n)=p(n)$ and $G_r(n)=0$ for $r\geqslant n+2$.

To our knowledge, the functions $S_r(n)$ and $G_r(n)$ have not been considered previously in the literature. 

It is known \cite[A022567]{Sloane} that the sum of the least gaps in all partitions of $n$ can be expressed in terms of the Euler's partition function $p(n)$:
\begin{equation}\label{eq0}
\sum_{k=0}^\infty p(n-T_k) = S_1(n),
\end{equation}
where $T_n= n(n+1)/2$ is the $n$-th triangular number.
Upon reflection, one expects that there might be infinite families of such identities where \eqref{eq0} is the first entry.
As far as we know, the following identity has not been remarked before.

\begin{theorem}\label{T0}
	For $n\geqslant 0$ and $r\geqslant 1$,
	\begin{equation}\label{eqT0}
	\sum_{k=0}^\infty p(n-rT_k) = S_r(n).
	\end{equation}
\end{theorem}

In section \ref{sec2}, we provide a combinatorial proof of Theorem \ref{T0}. 
Then, theta identity \eqref{eq1} and Theorem \ref{T0} allow us to find the generating function for $S_r(n)$ and to prove the  following result that also involves the partitions of $n$ into parts not congruent to $0$, $r$ or $3r \bmod {4r}$. We denote by $U_r(n)$ the number of these partitions.

\begin{theorem}\label{T1}
	For $n\geqslant 0$ and $r\geqslant 1$,
	\begin{enumerate}
		\item[(i)] $\displaystyle{ \sum_{k=0}^\infty \left(p(n-rT_{4k})+p(n-rT_{4k+3}) \right) = \frac{S_r(n)}{2}+\frac{U_r(n)}{2} ;}$
		\item[(ii)] $\displaystyle{ \sum_{k=0}^\infty \left(p(n-rT_{4k+1})+p(n-rT_{4k+2}) \right) = \frac{S_r(n)}{2}-\frac{U_r(n)}{2} .}$		
	\end{enumerate}
\end{theorem}

By this theorem, we see that $S_r(n)$ and $U_r(n)$ have the same parity. 

\begin{corollary}
	For $n\geqslant 0$ and $r\geqslant 1$,  the sum of the least $r$-gaps in all partitions of $n$ and the number of partitions of $n$ into parts not congruent to $0$, $r$ or $3r \bmod {4r}$ have the same parity.
\end{corollary}

In addition, we have the following identity.

\begin{corollary}\label{Cor1}
	For $n\geqslant 0$ and $r\geqslant 1$,
	$$\sum_{k=0}^\infty (-1)^{T_k} p(n-rT_k) = U_r(n).$$
\end{corollary}

Replacing $r$ by $1$ in Corollary \ref{Cor1}, we obtain another known identity (see \cite[the proof of Theorem 2.3]{BM}).

\begin{corollary}\label{C1}
	For $n\geqslant 0$,
	$$ \sum_{k=0}^\infty (-1)^{T_k} p(n-T_k) = 
			\begin{cases}
			q\left(\frac{n}{2} \right), & \text{for $n$ even,}\\
			0, & \text{for $n$ odd,}
			\end{cases}	
	$$
	 where $q(n)$ is the number of partitions of $n$ into distinct parts.	
\end{corollary}

It is shown in \cite[Corollary 4.7]{Merca16} that $q(n)$ is odd if and only if $n$ is a generalized pentagonal number. Thus,  we deduce the following result related to the parity of $S_1(n)$.

\begin{corollary}\label{C6}
	For $n\geqslant 0$, the sum of the least gaps in all partitions of $n$ is even except when $n$ is twice a generalized pentagonal number.
\end{corollary}

If $s\geqslant 3$ is the number of sides of a polygon, the $n$th $s$-polygonal number (or $s$-gonal number) is $$P(s,n)=\frac{n^2(s-2)-n(s-4)}{2}.$$ If we allow $n \in \mathbb Z$, we obtain generalized $s$-gonal numbers. Note that, for $n>0$, we have $P(3,-n)=P(3,n-1)$ and for all $n$ we have $P(4,-n)=P(4,n)$.  For $s \geqslant 5$ and $n >0$, $P(s,-n)$ is not an ordinary $s$-gonal number. We remark that the $n$-th $s$-gonal number can be expressed in term of the triangular numbers $T_n$ as follows:
$$P(s,n)=(s-3)T_{n-1}+T_n.$$

Beside Theorem \ref{T0}, there is another infinite family of identities involving Euler's partition function $p(n)$ for which \eqref{eq0} is the special case $r=1$. 

\begin{theorem}\label{T2}
	For $n\geqslant 0$ and $r\geqslant 1$
	\begin{equation}\label{pn}\sum_{k=0}^{\infty}p\left( n-P(r+2,-k)\right) =S_r(n)+G_r(n).\end{equation}
\end{theorem}

In this paper, we provide a purely combinatorial proof of this result and some applications involving partitions into even numbers of parts, partitions with nonnegative rank, and partitions with nonnegative crank.

\section{Combinatorial proof of Theorem \ref{T0}} \label{sec2}

Fix  $r\geqslant 1$ and, for each $k\geqslant 0$ consider the fat staircase partition (written in exponential notation) $$\delta_r(k)=(1^r, 2^r, \ldots, (k-1)^r, k^{r}).$$ This is the staircase partition with largest part $k$ in which each part is repeated $r$ times. Its size is equal to $rT_k$.

As before, fix $r\geqslant 1$ and also fix $n \geqslant 0$. For each $k\geqslant 0$  we create an injection from the set of partitions of $n-rT_k$ into the set of partitions of $n$ $$\varphi_{r,n,k}:\{\m \vdash n-rT_k\} \hookrightarrow \{\l \vdash n\}$$ where $\varphi_{r,n,k}(\m)$ is the partition obtained from $\m$ by inserting the parts of the staircase $\delta_r(k)$. Denote by $\A_{r,n,k}$ the image of $\{\m \vdash n-rT_k\}$ under $\varphi_{r,n,k}$. Thus, $p(n-rT_k)=|\A_{r,n,k}|$ and  $\A_{r,n,k}$ consists precisely of the partitions $\l$ of $n$ satisfying $g_r(\l)> k$.

Consider an arbitrary partition $\l$ of $n$ with  $g_r(\l)=k$. Then $\l\in \A_{r,n,i}$, $i=0, 1, \ldots k-1$ and  $\l \not \in  \A_{r,n,j}$ with $j\geqslant k$.  Therefore, each partition of $n$ with $g_r(\l)=k$ is counted by the left hand side of \eqref{eqT0} exactly $k$ times.

\section{Proof of Theorem \ref{T1}}

We rewrite the identity \eqref{eq1} as 
\begin{equation}\label{eq2}
\sum_{k=0}^{\infty} (-q)^{T_k} = \frac{(q;q)_\infty}{(q^2;q^4)_\infty}.
\end{equation}
Applying bisection on \eqref{eq2}, we obtain:
\begin{equation}\label{eq3}
\frac{1}{2} \sum_{k=0}^\infty \left(q^{T_k} \pm (-q)^{T_k}\right)  = \frac{1}{2} \frac{(-q;-q)_\infty \pm (q;q)_\infty}{(q^2;q^4)_\infty}.
\end{equation}
Multiplying both sides of \eqref{eq3} by the reciprocal of $(q;q)_\infty$, we give
\begin{align*}
\frac{1}{2(q;q)_\infty} \sum_{k=0}^\infty \left( q^{T_k} \pm (-q)^{T_k}\right) 
& = \frac{(-q^2;q^2)_\infty}{2} \frac{(-q;-q)_\infty \pm (q;q)_\infty}{(q;q)_\infty} \nonumber \\
& = \frac{(-q^2;q^2)_\infty}{2} \left(\frac{(-q;-q)_\infty}{(q;q)_\infty} \pm 1\right)\nonumber \\
& = \frac{(-q;q)^2_\infty \pm (-q^2;q^2)_\infty}{2}.
\end{align*}
By this identity, with $q$ replaced by $q^r$, we obtain the relation
\begin{equation*}
\frac{1}{2(q^r;q^r)_\infty} \sum_{k=0}^\infty \left( q^{rT_k} \pm (-q^r)^{T_k}\right)=
\frac{(-q^r;q^r)^2_\infty \pm (-q^{2r};q^{2r})_\infty}{2},
\end{equation*}
that can be rewritten as
\begin{align}
&\frac{1}{2(q;q)_\infty} \sum_{k=0}^\infty \left( q^{rT_k} \pm (-q^r)^{T_k}\right) \nonumber\\
&\qquad = \frac{1}{2} \left( \frac{(-q^r;q^r)^2_\infty (q^r;q^r)_\infty}{(q;q)_\infty} \pm \frac{(-q^{2r};q^{2r})_\infty (q^r;q^r)_\infty}{(q;q)_\infty}  \right) \nonumber\\
&\qquad = \frac{1}{2} \left( \frac{(-q^{r};q^{r})_\infty (q^{r};q^{r})_\infty}{(q;q)_\infty (q^r;q^{2r})_\infty} \pm \frac{ (q^r,q^{2r};q^{2r})_\infty}{(q;q)_\infty (q^{2r};q^{4r})_\infty}  \right) \nonumber \\
&\qquad = \frac{1}{2} \left( \frac{(q^{2r};q^{2r})_\infty}{(q;q)_\infty (q^r;q^{2r})_\infty} \pm \frac{ (q^r,q^{3r},q^{4r};q^{4r})_\infty}{(q;q)_\infty}  \right). \label{eq4}
\end{align}
Considering the generating function for $p(n)$, i.e.,
$$\sum_{n=0}^\infty p(n) q^n =\frac{1}{(q;q)_\infty}$$
and the theta identity \eqref{eq1}, by  Theorem \ref{T0} we deduce that
$$\sum_{k=0}^\infty S_r(k) q^k = \frac{(q^{2r};q^{2r})_\infty}{(q;q)_\infty (q^r;q^{2r})_\infty}.$$
On the other hand, we have
$$ \sum_{k=0}^\infty U_r(k) q^k = \frac{ (q^r,q^{3r},q^{4r};q^{4r})_\infty}{(q;q)_\infty}.$$
Taking into account  the well-known Cauchy multiplication of two power series, 
we deduce our identities as combinatorial interpretations of \eqref{eq4}.

\section{Combinatorial proof of Theorem \ref{T2}}

The proof of Theorem \ref{T2} is analogous to the proof of Theorem \ref{T0}. 

For fixed $r\geqslant 1$ and, for each $k\geqslant 0$we denote by ' $$\delta_r(k)=(1^r, 2^r, \ldots, (k-1)^r, k^{r-1})$$ the staircase partition in which the largest part is $k$ and is repeated $r-1$ times and all other parts are repeated $r$ times. Its size is equal to $P(r+2,-k)$.

As before, fix $r\geqslant 1$ and also fix $n \geqslant 0$. For each $k\geqslant 0$  we create an injection from the set of partitions of $n-P(r+2,-k)$ into the set of partitions of $n$ $$\varphi'_{r,n,k}:\{\m \vdash n-P(r+2,-k)\} \hookrightarrow \{\l \vdash n\}$$ where $\varphi'_{r,n,k}(\m)$ is the partition obtained from $\m$ by inserting the parts of the staircase $\delta'_r(k)$. If  $\A'_{r,n,k}$ denotes the image of $\{\m \vdash n-P(r+2,-k)\}$ under $\varphi_{r,n,k}$, we have that  $p(n-P(r+2,-k))=|\A'_{r,n,k}|$ and  $\A'_{r,n,k}$ consists precisely of the partitions $\l$ of $n$ satisfying $g_r(\l)\geqslant k$ and $g_{r-1}(\l)> k$.

If $\l \vdash n$  has  $g_r(\l)=k$, then $\l\in \A'_{r,n,i}$, $i=0, 1, \ldots k-1$. If $g_{r-1}(\l)=k$, then $\l \not \in  \A'_{r,n,j}$ with $j\geqslant k$. If $g_{r-1}>k$, then $\l  \in  \A'_{r,n,k}$ but $\l \not \in  \A'_{r,n,j}$ with $j> k$. Therefore, each partition of $n$ with $g_r(\l)=k$ is counted by the left hand side of \eqref{pn} exactly $k$ times if $g_r(\l)=g_{r-1}(\l)$ and exactly $k+1$ times if $g_r(\l)<g_{r-1}(\l)$.

\section{Applications of Theorem \ref{T2}}

In this section we consider some special cases of Theorem \ref{T2} in order to discover and prove new identities involving Euler's partition function $p(n)$.

\subsection{Partitions into even numbers of parts}

Now we consider the following classical theta identity \cite[eq. 2.2.12]{Andrews76}
\begin{equation}
\frac{(q;q)_\infty}{(-q;q)_\infty} =1+2 \sum_{k=1}^\infty (-1)^k q^{k^2}.
\end{equation}
Elementary techniques in the theory of  partition \cite{Andrews76} allow us to derive a known combinatorial interpretation of this identity, namely
\begin{equation}\label{eq11}
p(n)+2\sum_{j=k}^n (-1)^k p(n-k^2) = p_e(n)-p_o(n),
\end{equation}
where $p_e(n)$ is the number of partitions of $n$ into even number of parts and 
$p_o(n)$ is the number of partitions of $n$ into odd number of parts.
Moreover, it is known that 
\begin{equation}\label{eq12}
p_e(n) = p(n)+\sum_{k=1}^n (-1)^k p(n-k^2)
\end{equation}
and
$$p_o(n) = -\sum_{k=1}^n (-1)^k p(n-k^2).$$
These relations can be considered a bisection of the identity \eqref{eq11}. 
Combining  identity \eqref{eq12} with the case $r=2$ of Theorem \ref{T2}, we derive the following result.

\begin{corollary}
	For $n\geqslant 0$,
	\begin{enumerate}
		\item[(i)] $\displaystyle{ \sum_{k=0}^\infty p\left( n-(2k)^2\right) = \frac{S_2(n)+G_2(n)+p_e(n)}{2};}$
		\item[(ii)] $\displaystyle{ \sum_{k=0}^\infty p\left( n-(2k+1)^2\right) = \frac{S_2(n)+G_2(n)-p_e(n)}{2}.}$		
	\end{enumerate}
\end{corollary}

\subsection{Partitions with nonnegative rank}

In 1944, Dyson \cite{Dyson} defined the rank of a partition as the difference between its largest part and  the number of its parts. 
Then he observed empirically that the partitions of $5n + 4$ (respectively $7n + 5$) form $5$ (respectively $7$) groups of equal size when sorted by their ranks modulo $5$ (respectively $7$). This interesting conjecture of Dyson was proved ten years later by Atkin and Swinnerton-Dyer \cite{Atkin}.
In this section, we denote by $R(n)$ the number of partitions of $n$ with nonnegative rank. 

It is known \cite[A064174]{Sloane} that the number of partitions of $n$ with nonnegative rank can be expressed in terms of Euler's partition function as follows:
\begin{equation}
R(n)=\sum_{k=0}^n (-1)^k p\left( n-k(3k+1)/2\right).
\end{equation}
Considering the case $r=3$ of Theorem \ref{T2}, we obtain the following result.

\begin{corollary}
	For $n\geqslant 0$,
	\begin{enumerate}
		\item[(i)] $\displaystyle{ \sum_{k=0}^\infty p\left( n-k(6k+1)\right) = \frac{S_3(n)+G_3(n)+R(n)}{2};}$
		\item[(ii)] $\displaystyle{ \sum_{k=0}^\infty p\left( n-(2k+1)(3k+2)\right) = \frac{S_3(n)+G_3(n)-R(n)}{2}.}$		
	\end{enumerate}
\end{corollary}

\subsection{Partitions with nonnegative crank}

Dyson \cite{Dyson} conjectured the existence of a crank function for partitions that would provide a combinatorial proof of Ramanujan's congruence modulo $11$. Forty-four years later, Andrews and Garvan \cite{Andrews88} successfully found such a function which yields a combinatorial explanation of Ramanujan congruences modulo $5$, $7$, and $11$. 
For a partition $\lambda$, let $l(\lambda)$ denote the largest part of $\lambda$, $\omega(\lambda)$ denote the number of $1$'s in $\lambda$, and $\mu(\lambda)$  denote the number of parts of $\lambda$ greater than $\omega(\lambda)$. The crank $c(\lambda)$ is defined by
$$
c(\lambda) =
\begin{cases}
l(\lambda), & \text{for $\omega(\lambda)=0$,}\\
\mu(\lambda)-\omega(\lambda),  & \text{for $\omega(\lambda)>0$.}
\end{cases}
$$ 
In this section, we denote by $C(n)$ the number of partitions of $n$ with nonnegative crank. 
 
 We known \cite[A064428]{Sloane} that the number of partitions of $n$ with nonnegative crank can be expressed in terms of Euler's partition function $p(n)$:
 \begin{equation}\label{crank}
 C(n)=\sum_{k=0}^\infty (-1)^k p(n-T_k).
 \end{equation}
We have the following result related to the parity of $C(n)$.

\begin{corollary}
	For $n\geqslant 0$, the number of partitions of $n$ with nonnegative crank is even except when $n$ is twice a generalized pentagonal number.
\end{corollary}

\begin{proof}
	Considering the case $r=1$ of Theorem \ref{T2} and the identity \eqref{crank}, we obtain
	$$\sum_{k=0}^\infty p(n-T_{2k}) = \frac {C(n)+S_1(n)}{2}.$$
	We see that the number of partitions of $n$ with nonnegative crank and the sum of the least gaps in all partitions of $n$ have the same parity. According to Corollary \ref{C6} the proof is finished.
\end{proof}

By the identity \eqref{crank} and Corollary \ref{Cor1}, we easily get two identities. 

 \begin{corollary}
 	For $n\geqslant 0$,
 	\begin{enumerate}
 		\item[(i)] $\displaystyle{ \sum_{k=0}^\infty (-1)^k p\left( n-T_{k+2\lfloor k/2 \rfloor}\right) = \frac{C(n)+U_1(n)}{2};}$
 		\item[(ii)] $\displaystyle{ \sum_{k=0}^\infty (-1)^k p\left( n-T_{k+2\lfloor k/2 \rfloor+2}\right) = \frac{C(n)-U_1(n)}{2}.}$		
 	\end{enumerate}
 \end{corollary}

\bigskip


\end{document}